\documentclass[12pt,twoside,reqno]{amsart}
\usepackage{amsmath}
\usepackage{amsfonts}
\usepackage{amssymb}
\usepackage{color}
\usepackage{mathrsfs}
\newcommand{\RR}{\mathbb R}

\newcommand{\mf}{\mathcal{F}}
\newcommand{\mmab}{\mathcal{M}_{\alpha,\beta}}
\newcommand{\mzab}{\mathcal{Z}_{\alpha,\beta}}
\newcommand{\beqn}{\begin{equation}}
\newcommand{\eeqn}{\end{equation}}
\newcommand{\bean}{\begin{eqnarray}}
\newcommand{\eean}{\end{eqnarray}}
\DeclareMathAlphabet{\mathpzc}{OT1}{pzc}{m}{it}
\newtheorem{theorem}{Theorem}[section]
\newtheorem{corollary}[theorem]{Corollary}
\newtheorem{lemma}[theorem]{Lemma}

\newtheorem{remark}[theorem]{Remark}
\numberwithin{equation}{section}
\begin{document}
\title{A phase-field approximation of the Willmore flow\\ with volume and area constraints} 
\thanks{Partially supported by FP7-IDEAS-ERC-StG Grant \#200947 (BioSMA) and the MIUR-PRIN
Grant 2008ZKHAHN ``Phase transitions, hysteresis and multiscaling''}
\author{Pierluigi Colli}
\address{Dipartimento di Matematica ``F.~Casorati'', Universit\`a di Pavia, Via Ferrata 1, I-27100~Pavia,~Italy}
\email{pierluigi.colli@unipv.it}
\author{Philippe Lauren\c cot}
\address{Institut de Math\'ematiques de Toulouse, CNRS UMR~5219, Universit\'e de Toulouse, F--31062 Toulouse Cedex 9, France} 
\email{laurenco@math.univ-toulouse.fr}
\date{\today}
%
\begin{abstract}
The well-posedness of a phase-field approximation to the Willmore flow with area and volume constraints is established when the functional approximating the area has no critical point satisfying the two constraints. The existence proof relies on the underlying gradient flow structure of the problem: the time discrete approximation is solved by a variational minimization principle. The main difficulty stems from the nonlinearity of the area constraint.
\end{abstract}

\maketitle

%
%
\pagestyle{myheadings}
\markboth{\sc{Pierluigi Colli and Philippe Lauren\c cot}}{\sc{A phase-field approximation of the Willmore flow with volume and area constraints}}
%
\section{Introduction}\label{sec:intro}

Biological cell membranes define the border between the interior of the cell and its surrounding medium and can be roughly described as a lipid bilayer in which several kinds of lipids are assembled and through which proteins diffuse. The size of the cell (a few microns) is typically much larger than the thickness of the membrane (a few nanometers) and a possible approach to model the geometric properties of the latter is to assume the membrane to be a two-dimensional embedded surface $\Sigma$ in $\mathbb{R}^3$ with a shape at equilibrium being determined by the Canham-Helfrich elastic bending energy 
\begin{equation}
\mathcal{E}_{CH}(\Sigma) := \int_\Sigma \left[ \frac{k}{2} \left( \mathcal{H} - \mathcal{H}_0 \right)^2 + \frac{k_g}{2}\ \mathcal{K} \right]\ dS\,, \label{i0}
\end{equation}
see, e.g., \cite{Ca09,CHM06,DLW04,Li91} and the references therein. Here, $\mathcal{H} := (h_1+h_2)/2$ is the arithmetic mean of the principal curvatures $h_1$ and $h_2$ of $\Sigma$ (scalar mean curvature), $\mathcal{K} := h_1 h_2$ is the product of its principal curvatures (Gau\ss{} curvature), $k$ and $k_g$ are the bending rigidity and the Gaussian curvature rigidity, respectively, and $\mathcal{H}_0$ denotes the spontaneous curvature which accounts for the asymmetry of the membrane. Let us mention here that, when $\mathcal{H}_0=k_g=0$ and $k=1$, the functional $\mathcal{E}_{CH}$ is nothing but the Willmore functional which is a well-known object in differential geometry \cite{Wi93}. Two natural geometric constraints come along with cell membranes: the inextensibility of the membrane fixes the total area while a volume constraint follows from its permeability properties \cite{CHM06,DLW06}. 

Recently, experimental results have shown evidence of dynamic instabilities in membranes, see, e.g., \cite{Ca09} and the references therein, and provided the impetus for the development of dynamical models. A first approach is to consider the gradient flow associated to the Canham-Helfrich functional which describes the time evolution of a family of (smooth) surfaces $(\Sigma(t))_{t\ge 0}$ and reads (in the simplified situation $k=1$ and $\mathcal{H}_0=k_g=0$)
\begin{equation}
\mathcal{V} = - \Delta_\Sigma\mathcal{H} - 2\ \mathcal{H} \left( \mathcal{H}^2 - \mathcal{K} \right)\,,  \label{i1}
\end{equation}
where $\mathcal{V}$ and $\Delta_\Sigma$ denote the normal velocity to $\Sigma$ and the Laplace-Beltrami operator on $\Sigma$, respectively. Volume and area constraints can also be included and result in an additional term in the right-hand side of \eqref{i1} of the form $\ell+m \mathcal{H}$, the parameters $\ell$ and $m$ being the Lagrange multipliers associated with the two constraints. The main drawback of this approach is that it requires to solve a highly nonlinear free boundary problem which is difficult to study analytically and costly to compute numerically. However, numerical schemes have been recently developed for geometric evolution equations such as \eqref{i1}, see, e.g., \cite{BGN08}. 

A well-known alternative to free boundary problems is the phase-field approach where the sharp interface $\Sigma(t)$ is replaced by a diffuse interface which is nothing but a thin neighbourhood of thickness $\varepsilon$ of the zero level set of an $\varepsilon$-dependent smooth function, the order parameter. For biological membranes, this approach has been developed in several recent papers with and without the volume and area constraints \cite{BM10,Ca09,CHM06,CHM07,DLW04,DLW06,LM00} and can be described as follows, still in the simplified situation $k=1$ and $\mathcal{H}_0=k_g=0$: let $W$ be a smooth double-well potential (for instance, $W(r)=(1-r^2)^2/4$) and for $\varepsilon>0$ and $w\in H^2(\Omega)$ define the free energy $E_\varepsilon[w]$ by
\begin{equation}
E_\varepsilon[w] := \frac{\varepsilon}{2}\ \int_\Omega \left( \Delta w - \frac{W'(w)}{\varepsilon^2} \right)^2\ dx\,, \label{i2}
\end{equation}
where $\Omega$ is the spatial domain which comprises the cells and their surrounding medium. The corresponding phase-field model is the gradient flow of $E_\varepsilon$ in $L^2$ and reads
\begin{equation}
\partial_t v_\varepsilon = \varepsilon\ \Delta\mu_\varepsilon - \frac{1}{\varepsilon}\ W''(v_\varepsilon)\ \mu_\varepsilon\,, \qquad \mu_\varepsilon = - \Delta v_\varepsilon + \frac{1}{\varepsilon^2}\ W'(v_\varepsilon)\,, \label{i3}
\end{equation}
 supplemented with suitable initial and boundary conditions. As $\varepsilon$ approaches $0$, the function $v_\varepsilon$ is close to the values $\pm 1$ in large regions of the domain separated by narrow transition layers of width $\varepsilon$ around the zero level set $\{ x\ :\ v_\varepsilon(t,x)=0 \}$ of $v_\varepsilon$ at time $t$. It is this time-dependent family of level sets which is expected to converge as $\varepsilon\to 0$ to a family of surfaces $(\Sigma(t))_{t\ge 0}$ evolving according to the geometric motion \eqref{i1}. Formal asymptotic expansions have been performed to check the consistency of the free energy \eqref{i2} and the phase-field model \eqref{i3} with \eqref{i0} and \eqref{i1}, respectively, in the limit $\varepsilon=0$ \cite{CHM06,DLRW05,LM00,Wa08}. While no rigorous justification seems to be available so far for the evolution problem, the relationship between the minimizers of the Canham-Helfrich functional \eqref{i0} (without constraints) and those of the free energy \eqref{i2} has been the subject of recent studies \cite{BM10,DG91,Mo05,RS06}. In fact, the existence of minimizers of the free energy \eqref{i2} follows by standard arguments from the theory of the calculus of variations. But the well-posedness of the evolution phase-field model \eqref{i3} for a fixed positive $\varepsilon$ seems to be less obvious and, as far as we know, it is a widely open topic, though the phase-field approach has been used quite extensively to perform numerical simulations of the dynamics of biological membranes \cite{CHM06,CHM07,DLW04,DLW06}. We are only aware of two contributions in that direction. On the one hand, the well-posedness of the phase-field model \eqref{i3} with $\varepsilon=1$ and a volume constraint fixing the average of $v$ has been shown in \cite{CL11}. On the other hand, the existence of a weak solution to a system coupling a phase-field model (similar to \eqref{i3} but with a convection term) with the Navier-Stokes equation is established in \cite{WXxx}, relaxing the volume and area constraints by a penalisation approach. Therefore, accounting for both volume and area constraints in the phase-field approach to biological membranes does not seem to have been considered yet and is the focus of this paper. Before describing precisely our result, we recall that, in the phase-field approximation, the volume and area conservations read \cite{CHM06,DLRW05}:
\begin{equation}
\int_\Omega v_\varepsilon(t,x)\ dx = \text{const.} \;\;\mbox{ and }\;\; F_\varepsilon[v_\varepsilon] := \int_\Omega \left( \frac{\varepsilon}{2}\ |\nabla v_\varepsilon(t,x)|^2 + \frac{1}{\varepsilon}\ W(v_\varepsilon(t,x)) \right)\ dx = \text{const.} \label{i4}
\end{equation}
Indeed, recall that, as $\varepsilon\to 0$, the functional $F_\varepsilon$ approximates the perimeter functional \cite{MM77}. 

The purpose of this paper is then to investigate the existence and uniqueness of the phase-field approximation to the geometric flow  \eqref{i1} when both volume and area are fixed as described in \eqref{i4}. We may then set $\varepsilon=1$ in the forthcoming analysis and define
\begin{align}
W(r) & := \frac{a}{4} \left( r^2 - 1 \right)^2\,, \qquad r\in\RR\,, \nonumber\\[0.2cm]
F[w] & := \frac{\|\nabla w\|_2^2}{2} + \int_\Omega W(w(x))\, dx\,, \qquad w\in H^1(\Omega)\,, \label{b0}
\end{align}
where $a$ is a given positive real number. The phase-field approximation of \eqref{i1} with fixed volume and area turns out to be
\begin{align}
\partial_t v & = \Delta \mu - W''(v)\ \mu + A + B\ \mu \,, \qquad (t,x)\in (0,\infty)\times\Omega\,, \label{pf1} \\
\mu & = - \Delta v + W'(v)  \,, \qquad (t,x)\in (0,\infty)\times\Omega\,, \label{pf2} \\
\nabla v\cdot \mathbf{n} & = \nabla \mu \cdot\mathbf{n} = 0\,, \qquad (t,x)\in (0,\infty)\times\partial\Omega\,, \label{pf3} \\
v(0) & = v_0\,, \qquad x\in\Omega\,, \label{pf4}
\end{align}
where $\mathbf{n}$ denotes the outward unit normal vector field to $\partial\Omega$ and $\overline{w}$ the spatial average of $w\in L^1(\Omega)$, that is,
$$
\overline{w} := \frac{1}{|\Omega|}\ \int_\Omega w(x)\ dx\,.
$$
In \eqref{pf1}, $A$ and $B$ are time-dependent functions and the Lagrange multipliers corresponding to the volume and area constraints
\begin{equation}
\overline{v(t)} = \overline{v_0}\;\; \mbox{ and }\;\; F[v(t)] = F[v_0]\,, \qquad t\ge 0\,. \label{pf5}
\end{equation} 
When the area constraint is not taken into account (which corresponds to take $B=0$ in \eqref{pf1} and keep only the first constraint in \eqref{pf5}), the well-posedness of the resulting version of \eqref{pf1}-\eqref{pf5} is shown in \cite{CL11}. In that case, the equation turns out to be a gradient flow for the free energy
\begin{equation}\label{b4}
E[w] := \frac{1}{2}\ \int_\Omega \left[ - \Delta w + W'(w) \right]^2\ dx\,, \qquad w\in H^2(\Omega)\,,
\end{equation}
and the existence proof exploits this structure and relies on a time-discrete minimization scheme. This gradient flow structure is still available for \eqref{pf1}-\eqref{pf5} with the main difference that we now have two constraints including the additional one, which is nonlinear and generates several new difficulties in the analysis of the minimizing scheme. In particular, we emphasize that the two constraints may not be linearly independent and this happens in particular for critical points of $F$ under a volume constraint. Deriving the Euler-Lagrange equation for the time-discrete minimization scheme is then not obvious, with the further drawback that the area constraint is nonlinear. This difficulty strikes back when we wish to estimate the Lagrange multiplier $B$ and we have to restrict our analysis to the case where critical points of $F$ under a volume constraint cannot be reached during time evolution. 

Let us now introduce some notations: given $\alpha\in\RR$ and $\beta\in [0,\infty)$, the fact that there exists at least one function $w\in H^1(\Omega)$ satisfying simultaneously $\overline{w}=\alpha$ and $F[w]=\beta$ is not granted and requires a compatibility condition which we describe now. We set 
\begin{equation}\label{b1}
\beta_\alpha := \inf{\left\{ F[w]\ :\ w\in H^1(\Omega)\,, \ \overline{w} = \alpha \right\}}\,,
\end{equation}
which is well-defined owing to the nonnegativity of $F$, and
\begin{align}
\mathcal{M}_{\alpha,\beta}^1 & := \left\{ w\in H^1(\Omega)\ :\ \overline{w} = \alpha \;\mbox{ and }\; F[w]=\beta \right\}\,, \label{b2} \\
\mathcal{M}_{\alpha,\beta}^2 & := \left\{ w\in H^2_N(\Omega)\ :\ w\in \mathcal{M}_{\alpha,\beta}^1 \right\}\,, \label{b2b}
\end{align}
where
$$
H_N^2(\Omega) := \{ w \in H^2(\Omega)\ :\ \nabla v\cdot \mathbf{n} = 0 \;\;\mbox{ on }\;\; \partial\Omega \}\,.
$$ 
Clearly, $\mathcal{M}_{\alpha,\beta}^1=\emptyset$ if $ \beta\in [0,\beta_\alpha)$, while we shall show below that $\mathcal{M}_{\alpha,\beta}^i$, $i=1,2$, is quite large if $\beta>\beta_\alpha$ (see Lemma~\ref{le:b1} below).

Finally, as already mentioned, we have to exclude some values of the parameters $(\alpha,\beta)$ for which there are critical points of $F$ under a volume constraint in $\mathcal{M}_{\alpha,\beta}^2$. To this end, we introduce the set $\mathcal{Z}_{\alpha,\beta}$ defined by
\begin{equation}\label{b3}
\mathcal{Z}_{\alpha,\beta} := \left\{ w\in \mathcal{M}_{\alpha,\beta}^2\ :\ -\Delta w + W'(w) - \overline{W'(w)} = 0 \;\mbox{ in }\; \Omega \right\}\,,
\end{equation}
and shall require this set to be empty, an assumption which is fulfilled if $\beta$ is sufficiently large compared to $|\alpha|$, see Lemma~\ref{le:b3} below. We may now state our result:

\begin{theorem}\label{th:a1}
Consider $\alpha\in\RR$ and $\beta\in (\beta_\alpha,\infty)$ such that 
\begin{equation}
\mzab = \emptyset\,. \label{a0}
\end{equation}
Given an initial condition $v_0\in\mmab^2$, there is a unique function 
\begin{equation*}
v\in\mathcal{C}([0,\infty)\times\overline{\Omega}) \cap L^\infty(0,\infty; H^2(\Omega))\,, \qquad v(0)=v_0\,, 
\end{equation*}
such that, for all $t>0$, 
\begin{equation*}
 v(t)\in\mmab^2\,, \qquad \mu := -\Delta v + W'(v) \in L^2(0,t; H_N^2(\Omega))\,,
\end{equation*}
and there are two functions $A\in L^2(0,t)$ and $B\in L^2(0,t)$ such that 
\begin{equation}
\partial_t v = \Delta\mu - W''(v)\ \mu + A + B\ \mu \;\;\mbox{ a.e. in }\;\; (0,t)\times\Omega\,. \label{a4}
\end{equation}
In fact, there hold
\begin{equation}
A + B\ \overline{\mu} = \overline{W''(v)\mu} \label{a1}
\end{equation}
and
\begin{equation}
B\ \left\| \mu - \overline{\mu} \right\|_2^2 = \|\nabla\mu\|_2^2 + \int_\Omega W''(v)\ \mu^2\ dx - \overline{W''(v)\mu}\ \int_\Omega\mu\ dx \,. \label{a2}
\end{equation}
Moreover, for all $t>0$ there exists $\varepsilon(t)>0$ such that
\begin{equation}
\|(\mu - \overline{\mu})(s)\|_2\ge \varepsilon(t)>0\,, \qquad s\in [0,t]\,. \label{a3}
\end{equation}
\end{theorem}

Observe that the identity \eqref{a2} defining $B$ is meaningless if $\mu(t_0)$ is a constant at some time $t_0$, that is, if $v(t_0)\in\mzab$. The main purpose of the assumption \eqref{a0} is then to prevent this situation to occur.

A further consequence of our analysis is the time monotonicity of the free energy along the flow which is a natural outcome of the gradient flow structure of \eqref{pf1}-\eqref{pf5}. 

\begin{corollary}\label{co:a2} Under the assumptions and notations of Theorem~\ref{th:a1}, the map 
$$
t\longmapsto E[v(t)]= \frac{1}{2}\ \|\mu(t)\|_2^2 \;\;\mbox{  is non-increasing.}
$$
\end{corollary}

The outline of the paper is as follows: in the next section, we collect some preliminary results concerning the structure of $\mmab^i$, $i=1,2$, some functional lower bounds on $F$ and $E$, and the fact that $\mzab$ is indeed empty for $\beta$ large enough, along with a useful functional inequality in that case. Section~\ref{sec:min} describes the minimizing scheme for one time step. Estimates are also derived there, allowing us to pass to the limit as the time step decreases to zero and obtain the existence part of Theorem~\ref{th:a1} in Section~\ref{sec:exist}. The uniqueness part of Theorem~\ref{th:a1} is proved in Section~\ref{sec:uniq} and the proof heavily relies on the positivity property \eqref{a3} which allows us to control the difference between the area Lagrange multipliers of the two solutions.

\section{Preliminaries}\label{sec:prep}

Let us first show that $\mathcal{M}_{\alpha,\beta}^i$, $i=1,2$, is quite large when $\alpha\in\RR$ and $\beta\in (\beta_\alpha,\infty)$ as claimed in the Introduction.

\begin{lemma}\label{le:b1}
Consider $\alpha\in\RR$. There is at least a function $w_\alpha\in H^2_N(\Omega)$ such that $\overline{w_\alpha}=\alpha$ and $F[w_\alpha]=\beta_\alpha$. In addition, if $\beta\in (\beta_\alpha,\infty)$ and $\varphi\in H^1(\Omega)$ (resp.~$\varphi\in H^2_N(\Omega)$) satisfies $\overline{\varphi}=0$ and $\varphi\not\equiv 0$, then there is $\lambda>0$ depending on $\varphi$ such that $w_\alpha + \lambda \varphi \in \mathcal{M}_{\alpha,\beta}^1$ (resp. $w_\alpha + \lambda \varphi \in \mathcal{M}_{\alpha,\beta}^2$).
\end{lemma}

\begin{proof}
Let $\alpha\in\RR$. The existence of $w_\alpha$ follows from the nonnegativity and weak lower semicontinuity of $F$ by classical arguments of the theory of the calculus of variations; moreover, $w_\alpha$ solves the Euler-Lagrange variational identity associated with $F$, that is, 
$$
\int_\Omega \nabla w_\alpha \cdot \nabla z \, dx + \int_\Omega 
\left( W'( w_\alpha  ) - \overline{W'( w_\alpha  ) }\right) z \, dx =0 
\quad \text{ for all }\;\; z\in H^1(\Omega), 
$$
whence $w_\alpha\in H^2_N(\Omega)$. Consider next $\beta\in (\beta_\alpha,\infty)$ and a function $\varphi$ in either $H^1(\Omega)$ or $H_N^2(\Omega)$ such that $\overline{\varphi}=0$ and $\varphi\not\equiv 0$. Introducing the function $f$ defined by $f(\lambda):= F[w_\alpha + \lambda \varphi]$ for $\lambda\ge 0$, we realize that $f$ is a continuous function in $[0,\infty)$ with $f(0)=F[w_\alpha]=\beta_\alpha<\beta$ and, thanks to the nonnegativity of $W$, 
$$
f(\lambda) \ge \frac{1}{2} \left( \|\nabla w_\alpha\|_2^2 + \lambda^2 \|\nabla\varphi\|_2^2 + 2\lambda \int_\Omega \nabla w_\alpha \cdot \nabla\varphi\ dx \right) \mathop{\longrightarrow}_{\lambda\to\infty} \infty\,. 
$$
The mean-value theorem then guarantees that there is at least $\lambda_\varphi>0$ such that $f(\lambda_\varphi)=\beta$, that is, $w_\alpha + \lambda_\varphi \varphi\in \mathcal{M}_{\alpha,\beta}^i$ for either $i=1$ or $i=2$.
\end{proof}

We next show that the functionals $F$ and $E$ control the $H^1$-norm and the $H^2$-norm, respectively.

\begin{lemma}\label{le:b2}
Given $\alpha\in\RR$, there is $C_1>0$ such that
\begin{align}
\|w\|_{H^1} & \le C_1 \left( 1 + \sqrt{F[w]} \right) \;\;\mbox{ for }\;\; w\in H^1(\Omega) \;\;\mbox{ such that }\;\; \overline{w}=\alpha\,, \label{b6} \\
\|w\|_{H^2} & \le C_1 \left( 1 + \sqrt{E[w]} \right) \;\;\mbox{ for }\;\; w\in H_N^2(\Omega) \;\;\mbox{ such that }\;\; \overline{w}=\alpha\,. \label{b7}
\end{align}
\end{lemma}

\begin{proof}
Let $w\in H^1(\Omega)$ with $\overline{w}=\alpha$. We readily infer from the Poincar\'e-Wirtinger inequality
\begin{equation}
\|w -\overline{w}\|_2 \le C_2\ \|\nabla w\|_2\,, \label{b8}
\end{equation}
and the nonnegativity of $W$ that 
\begin{align*}
\|w\|_{H^1}^2 & \le 2 \left( \|w -\overline{w}\|_2^2 + \|\overline{w}\|_2^2 \right) + \|\nabla w\|_2^2 \le (1+2C_2)\ \|\nabla w\|_2^2 + 2\alpha^2 |\Omega| \\
& \le (2+4C_2)\ F[w] + 2\alpha^2 |\Omega|\,,
\end{align*}
whence \eqref{b6} follows. Next, let $w\in H_N^2(\Omega)$ with $\overline{w}=\alpha$ and set $\mu:= - \Delta w + W'(w)$. The definition \eqref{b4} of $E$ entails $\|\mu\|_2^2=2E[w]$ so that $\mu\in L^2(\Omega)$ and
$$
\|w-\alpha\|_2\ \|\mu\|_2 \ge \int_\Omega (w-\alpha)\ \mu\ dx  = \|\nabla w\|_2^2 + a\ \int_\Omega (w-\alpha)\ (w^3 - w)\ dx \ge \|\nabla w\|_2^2 -C\,.
$$
Thanks to \eqref{b8} and Young's inequality, we further obtain
$$
\|\nabla w\|_2^2  \le C + C_2\ \|\nabla w\|_2\ \|\mu\|_2 \le \frac{\|\nabla w\|_2^2}{2} + C \left(  1 + \|\mu\|_2^2 \right) \,,
$$
and consequently
$$
\|\nabla w\|_2^2  \le C (1 + E[w])\,.
$$
Using once more \eqref{b8} gives
\begin{equation}
\|w\|_{H^1}^2 \le C \left( 1 + E[w] \right)\,. \label{b9}
\end{equation}
Finally, observing that $r\mapsto W'(r) + a r$ is non-decreasing, it follows from the definition of $\mu$ that $w$ solves 
$$
-\Delta w + W'(w) + a w = \mu + a w \;\;\mbox{ in }\;\; \Omega
$$
with homogeneous Neumann boundary conditions, and a classical monotonicity argument ensures that
$$
\|\Delta w \|_2 \le \|\mu + a w \|_2 \le \|\mu\|_2 + a\ \|w\|_2\,.
$$
Combining this estimate with \eqref{b9} readily gives \eqref{b7}.
\end{proof}

The last result of this section is devoted to the set $\mathcal{Z}_{\alpha,\beta}$ defined in \eqref{b3}. We prove in particular another fact claimed in the Introduction, namely that, given $\alpha\in\RR$, the set $\mathcal{Z}_{\alpha,\beta}$ is empty at least for $\beta$ large enough, so that Theorem~\ref{th:a1} can be applied in that case. Throughout the paper, we use the following notation: given $w\in L^2(\Omega)$ satisfying $\overline{w}=0$, the function $\mathcal{N}(w)\in H_N^2(\Omega)$ is the unique solution to 
\begin{equation}
-\Delta\mathcal{N}(w) = w \;\;\mbox{ in }\;\;\Omega\,, \qquad \nabla\mathcal{N}(w)\cdot\mathbf{n} = 0 \;\;\mbox{ on }\;\;\partial\Omega\,, \quad\mbox{ satisfying }\;\; \overline{\mathcal{N}(w)}=0\,. \label{b10}
\end{equation}

\begin{lemma}\label{le:b3}
Consider $\alpha\in\RR$ and set $\mathcal{O}_\alpha := \{\beta \in (\beta_\alpha,\infty)\ :\ \mathcal{Z}_{\alpha,\beta}=\emptyset\}$. 
\begin{enumerate}
\item The set $\mathcal{O}_\alpha$ is open and $\beta\in\mathcal{O}_\alpha$ if $\beta$ is large enough.
\item Assume that $\beta\in\mathcal{O}_\alpha$. Given $M>0$, it turns out that
\begin{equation}
m_M := \inf\left\{ \left\| \nabla\mathcal{N}\left( -\Delta v + W'(v) - \overline{W'(v)} \right) \right\|_2^2\ :\ v\in \mathcal{M}_{\alpha,\beta}^2\,, \ E[v]\le M \right\} > 0\,. \label{b11}
\end{equation}
\end{enumerate}
\end{lemma}

\begin{proof} (1) Let us first show that $(\beta_\alpha,\infty)\setminus \mathcal{O}_\alpha$ is closed. Consider a sequence $(\beta_n)_{n\ge 1}$ in $(\beta_\alpha,\infty)$ such that $\beta_n\not\in\mathcal{O}_\alpha$ for each $n\ge 1$ and $\beta_n\to\beta \in (\beta_\alpha,\infty)$ as $n\to\infty$. Then, according to the definition of $\mathcal{O}_\alpha$, for each $n\ge 1$, there is $w_n\in \mathcal{M}_{\alpha,\beta_n}^2$ such that $-\Delta w_n + W'(w_n) - \overline{W'(w_n)} = 0$ in $\Omega$. Since $(\beta_n)_{n\ge 1}$ is bounded, it follows from Lemma~\ref{le:b2} that $(w_n)_{n\ge 1}$ is bounded in $H^1(\Omega)$. Furthermore, the properties of $w_n$ and the continuous embedding of $H^1(\Omega)$ in $L^3(\Omega)$ entail that
$$
E[w_n] = 2 \left\| \overline{W'(w_n)} \right\|_2^2 \le C \left( 1 + \|w_n\|_3^6 \right) \le C \left( 1 + \|w_n\|_{H^1}^6 \right)\,,
$$
so that $(E[w_n])_{n\ge 1}$ is bounded. A further application of Lemma~\ref{le:b2} ensures that $(w_n)_{n\ge 1}$ is bounded in $H^2(\Omega)$. We then deduce from the compactness of the embedding of $H_N^2(\Omega)$ in $H^1(\Omega)\cap\mathcal{C}(\overline{\Omega})$ that there are $w\in H_N^2(\Omega)$ and a subsequence $(w_{n_k})_{k\ge 1}$ of $(w_n)_{n\ge 1}$ such that 
\begin{equation*}
w_{n_k} \longrightarrow w \;\;\mbox{ in }\;\; H^1(\Omega) \;\;\mbox{ and }\;\; \mathcal{C}(\overline{\Omega})\,, \qquad
w_{n_k} \rightharpoonup w \;\;\mbox{ in }\;\; H^2(\Omega) \,.
\end{equation*}
It is then straightforward to check that $w\in \mathcal{M}_{\alpha,\beta}^2$ and satisfies $-\Delta w + W'(w) - \overline{W'(w)} = 0$ in $\Omega$, that is, $w\in\mathcal{Z}_{\alpha,\beta}$. Thus, $\beta\not\in\mathcal{O}_\alpha$.

Consider next $v\in\mathcal{Z}_{\alpha,\beta}$. Then, on the one hand, 
we have
\begin{align*}
\int_\Omega v \left( -\Delta v + W'(v) \right)\ dx &= \|\nabla v\|_2^2 + a\ \int_\Omega v^2 \left( v^2 - 1 \right)\ dx = 2 F(v) + \frac{a}{2}\ \int_\Omega (v^4 - 1)\ dx  \\
&= 2\beta + \frac{a}{2} \left(  \|v\|_4^4 - |\Omega| \right) \,.
\end{align*}
On the other hand, it results that
\begin{equation*}
 \int_\Omega v \left( -\Delta v + W'(v) \right)\ dx  = \overline{W'(v)}\ \int_\Omega v\ dx = \alpha a\ \int_\Omega (v^3-v)\ dx\,.
\end{equation*}
The above two identities give 
$$
2\beta + \frac{a |\Omega|}{2} \left( 2\alpha^2 - 1 \right) + \frac{a}{2}\ \int_\Omega \left( v^4 - 2\alpha\ v^3 \right)\ dx = 0\,.
$$
Since $r^4 - 2 \alpha r^3 \ge - 27 \alpha^4/16$ for $r\in\RR$, we deduce that 
$$
2\beta + \frac{a |\Omega|}{2} \left( 2\alpha^2 - 1 - \frac{27}{16}\ \alpha^4\right) \le 0\,,
$$
which is not possible if $\beta$ is large enough.

(2) Assume for contradiction that $m_M=0$ and let $(v_k)_{k\ge 1}$ be a minimizing sequence. Since $v_k\in\mathcal{M}_{\alpha,\beta}^2$ with $E[v_k]\le M$, it follows from Lemma~\ref{le:b2} that $(v_k)_{k\ge 1}$ is bounded in $H_N^2(\Omega)$ and thus compact in $H^1(\Omega)$ and $\mathcal{C}(\overline{\Omega})$. Therefore, there are $v\in H_N^2(\Omega)$ and a subsequence of $(v_k)_{k\ge 1}$ (not relabeled) such that 
\begin{equation*}
v_k \longrightarrow v \;\;\mbox{ in }\;\; H^1(\Omega) \;\;\mbox{ and }\;\; \mathcal{C}(\overline{\Omega})\,, \qquad
v_k \rightharpoonup v \;\;\mbox{ in }\;\; H^2(\Omega) \,.
\end{equation*}
These convergences readily imply that $\overline{v}=\alpha$ and $F[v]=\beta$, so that $v\in\mathcal{M}_{\alpha,\beta}^2$. In addition,  setting $\nu_k:=-\Delta v_k + W'(v_k)- \overline{W'(v_k)}$ for $k\ge 1$, we also deduce that $(\nu_k)_{k\ge 1}$ converges weakly in $L^2(\Omega)$ towards $\nu := -\Delta v + W'(v) - \overline{W'(v)}$ while the property $m_M=0$ entails that $(\nu_k)_{k\ge 1}$ converges to zero in $H^1(\Omega)'$. Consequently, $\nu=0$, from which we conclude that $v\in\mathcal{Z}_{\alpha,\beta}$ and get a contradiction.
\end{proof}

\section{The minimizing scheme}\label{sec:min}

We fix $\alpha\in\RR$ and $\beta\in (\beta_\alpha,\infty)$. Given $f\in L^2(\Omega)$ and $\tau\in (0,1)$, we introduce the functional $\mf_{\tau,f}$ on $H_N^2(\Omega)$ defined by 
\begin{equation}
\mf_{\tau,f}[w] := \frac{\|w-f\|_2^2}{2} + \tau\ E[w]\,, \qquad w\in H_N^2(\Omega)\,, \label{c0}
\end{equation}
and consider the following minimization problem
\begin{equation}
\omega_{\tau,f} := \inf\left\{ \mf_{\tau,f}[w]\ :\ w\in\mmab^2 \right\}\,. \label{c1}
\end{equation}
Since $E$ is non-negative and $\mmab^2$ is non-empty by Lemma~\ref{le:b1}, $\omega_{\tau,f}$ is well-defined and non-negative.

\begin{lemma}\label{le:c1}
The functional $\mf_{\tau,f}$ has at least a minimizer in $\mmab^2$. In addition, any minimizer $v$ of $\mf_{\tau,f}$ in $\mmab^2$ satisfies
\begin{equation}
E[v]\le E[f]\,. \label{c1b}
\end{equation}
\end{lemma}
 
\begin{proof}
For each $k\ge 1$, there is $v_k\in\mmab^2$ such that 
\begin{equation}\label{c2}
\omega_{\tau,f} \le \mf_{\tau,f}[v_k] = \frac{\|v_k-f\|_2^2}{2} + \tau\ E[v_k] \le \omega_{\tau,f} + \frac{1}{k}\,.
\end{equation}
We introduce $\mu_k:=-\Delta v_k + W'(v_k)$ for $k\ge 1$. 
From Lemma~\ref{le:b2} and \eqref{c2} it follows that $(v_k)_{k\ge 1}$ is bounded in $H_N^2(\Omega)$. Since $H^2(\Omega)$ is compactly embedded in $H^1(\Omega)$ and $\mathcal{C}(\overline{\Omega})$, there are $v\in H_N^2(\Omega)$ and a subsequence of $(v_k)_{k\ge 1}$ (not relabeled) such that 
\begin{equation*}
v_k \longrightarrow v \;\;\mbox{ in }\;\; H^1(\Omega) \;\;\mbox{ and }\;\; \mathcal{C}(\overline{\Omega})\,, \qquad
v_k \rightharpoonup v \;\;\mbox{ in }\;\; H^2(\Omega) \,.
\end{equation*}
These convergences readily imply that $\overline{v}=\alpha$, $F[v]=\beta$, and $\mu_k\rightharpoonup \mu:=-\Delta v + W'(v)$ in $L^2(\Omega)$. Consequently, $v\in\mmab^2$ and 
\begin{align*}
\omega_{\tau,f} \le \mf_{\tau,f}[v] = \frac{\|v-f\|_2^2}{2} + 2\tau\ \|\mu\|_2^2 & \le  \lim_{k\to\infty} \, \frac{\|v_k-f\|_2^2}{2}  + \liminf_{k\to\infty} \, 2\tau\ \|\mu_k\|_2^2   \\
& = \liminf_{k\to\infty} \mf_{\tau,f}[v_k] = \omega_{\tau,f}\,.
\end{align*}
We have thus established that $v$ is a minimizer of $\mf_{\tau,f}$ in $\mmab^2$.
\end{proof}

\begin{remark}
\label{ven}
A similar argument shows that the free energy $E$ has at least a minimizer in $\mmab^2$. Constructing a minimizer to the Canham-Helfrich functional 
\eqref{i0} turns out to be far more complicated even without constraints
(see~\cite{Ri08, Si93} and references therein). Existence of axisymmetric minimizers to this functional with volume and area constraints has been recently proved in~\cite{CVxx}.
\end{remark}

The next step is to derive the Euler-Lagrange equation corresponding to the minimization problem \eqref{c1}. At this point, a new difficulty shows up as the constraints are not always independent. Indeed, setting $I[w]:=\overline{w}$, the differentials of $I$ and $F$ are $DI[w] = 1/|\Omega|$ and $DF[w] = - \Delta w + W'(w)$ and are not linearly independent if $-\Delta w + W'(w)$ is a constant. Assuming that this is not the case, we have the following result:

\begin{lemma}\label{le:c2}
Assume that $v\in\mmab^2$ solves the minimization problem \eqref{c1} and is such that $\mu:= -\Delta v + W'(v)$ is not a constant. Then $\mu\in H_N^2(\Omega)$ and there are real numbers $A$ and $B$ such that
\begin{equation}
\frac{v-f}{\tau} - \Delta\mu + W''(v)\ \mu = A + B\ \mu \;\;\mbox{ in }\;\; \Omega\,. \label{c17}
\end{equation}
\end{lemma}
 
\begin{proof} Owing to the nonlinear constraint $F[w]=\beta$, we proceed as in \cite[Proposition~43.6]{Ze85}. We observe that, since $\mu$ is not a constant, 
\begin{equation}
\langle \mu , \mu-\overline{\mu} \rangle_2 = \|\mu - \overline{\mu}\|_2^2 > 0\,. \label{c12}
\end{equation}
Next, it turns out that $\mu-\overline{\mu}$ is not sufficiently smooth for the forthcoming analysis and a regularization is needed. Owing to the density of $\mathcal{C}_0^\infty(\Omega)$ in $L^2(\Omega)$, for $\eta\in (0,1)$, there exists $\nu_\eta\in \mathcal{C}_0^\infty(\Omega)$ such that
\begin{equation}
\overline{\nu_\eta}=0 \;\;\mbox{ and }\;\; \|\nu_\eta - (\mu - \overline{\mu}) \|_2\le\eta\,.\label{c12b}
\end{equation}
Combining \eqref{c12} and \eqref{c12b} gives
$$
\langle \mu , \nu_\eta \rangle_2 = \langle \mu - \overline{\mu} , \nu_\eta \rangle_2 = \|\mu - \overline{\mu}\|_2^2 + \langle \mu - \overline{\mu} , \nu_\eta - (\mu - \overline{\mu} )\rangle_2 \ge \|\mu - \overline{\mu}\|_2 \left( \| \mu - \overline{\mu} \|_2 - \eta \right)\,,
$$
whence
\begin{equation}
\langle \mu , \nu_\eta \rangle_2 \ge \frac{\|\mu-\overline{\mu} \|_2^2}{2}>0 \label{c12c}
\end{equation}
for $\eta$ sufficiently small.

We now fix $\eta\in (0,1)$ small enough such that \eqref{c12c} holds true, let $n\ge 1$ and take $\zeta\in H_N^2(\Omega)\cap\mathcal{C}^\infty(\overline{\Omega})$ with $\overline{\zeta}=0$: we aim at constructing a perturbation of $v$ in $\mmab^2$. To this end, we define
\begin{equation*}
\gamma_\eta(\zeta) := -\frac{\langle \mu , \zeta \rangle_2}{\langle \mu , \nu_\eta \rangle_2}\;\;\mbox{ and }\;\;
\varphi_\lambda(t) := F\left[ v+\lambda(\gamma_\eta(\zeta)+t)\ \nu_\eta + \lambda\ \zeta \right] - \beta
\end{equation*}
for $\lambda>0$ and $t\in [-1/n,1/n]$. Since
\begin{align*}
\|(\gamma_\eta(\zeta)+t)\ \nu_\eta + \zeta\|_2 &\le \left( \frac{|\langle \mu , \zeta \rangle_2|}{\langle \mu , \nu_\eta \rangle_2} + \frac{1}{n} \right)\ \|\nu_\eta\|_2 + \|\zeta\|_2 \\
&\le \left( 1 + 2\ \frac{|\langle \mu , \zeta \rangle_2|}{\|\mu-\overline{\mu}\|_2^2}  \right)\ (1+\|\mu-\overline{\mu}\|_2) + \|\zeta\|_2
\end{align*}
for all $n\ge 1$ and $t\in [-1/n,1/n]$, we have 
$$
\varphi_\lambda(t) = F[v]-\beta + \lambda\ \left\langle \mu , \left[ (\gamma_\eta(\zeta)+t)\ \nu_\eta + \zeta \right] \right\rangle_2 + \lambda\ \varepsilon_\eta(\lambda)\,, 
$$
where $\varepsilon_\eta$ is a function which depends neither on $n\ge 1$ nor on $t\in [-1/n,1/n]$ and satisfies $\varepsilon_\eta(\lambda)\to 0$ as $\lambda\to 0$. Owing to the property $v\in\mmab^2$ and the definition of $\gamma_\eta(\zeta)$, we find
$$
\varphi_\lambda\left( \pm \frac{1}{n} \right) = \pm \frac{\lambda}{n}\ \langle \mu , \nu_\eta \rangle_2 + \lambda\ \varepsilon_\eta(\lambda)\,, 
$$
and it follows from the positivity \eqref{c12c} of $\langle \mu , \nu_\eta \rangle_2$ that there exists $\lambda_n\in (0,1/n)$ such that 
$$
\varphi_{\lambda_n}\left( - \frac{1}{n} \right) < 0 < \varphi_{\lambda_n}\left( + \frac{1}{n} \right) \,.
$$
The mean-value theorem then guarantees that there is $t_n\in (-1/n,1/n)$ such that $\varphi_{\lambda_n}(t_n)=0$. Setting $\zeta_n := (\gamma_\eta(\zeta)+t_n)\ \nu_\eta + \zeta$, we have thus shown that $F[v+\lambda_n \zeta_n]= \beta$. Since 
$$
\overline{v+\lambda_n \zeta_n} = \overline{v} + \lambda_n (\gamma_\eta(\zeta)+t_n)\ \overline{\nu_\eta} + \lambda_n\ \overline{\zeta} = \alpha
$$
by \eqref{c12b} and the properties of $\zeta$, we conclude that $v+\lambda_n \zeta_n\in\mmab^1$. Recalling that $v$, $\nu_\eta$, and $\zeta$ belong to $H_N^2(\Omega)$, so that $v+\lambda_n \zeta_n\in\mmab^2$, the minimizing property of $v$ ensures that $\mf_{\tau,f}[v]\le \mf_{\tau,f}[v+\lambda_n \zeta_n]$, that is
$$
\frac{\|v-f\|_2^2}{2} + \tau\ E[v] \le \frac{\|v+\lambda_n \zeta_n-f\|_2^2}{2} + \tau\ E[v+\lambda_n \zeta_n]\,.
$$
This inequality also reads
\begin{align}
0 \le & - \big\langle \Delta v , W'(v+\lambda_n \zeta_n) - W'(v) \big\rangle_2 - \lambda_n\ \big\langle \Delta\zeta_n ,  W'(v+\lambda_n \zeta_n) - \Delta v \big\rangle_2 \nonumber \\
& + \frac{1}{2}\ \big\langle W'(v+\lambda_n \zeta_n) - W'(v) , W'(v+\lambda_n \zeta_n) + W'(v) \big\rangle_2 \nonumber \\
& + \frac{\lambda_n^2}{2}\ \|\Delta\zeta_n\|_2^2 + \frac{\lambda_n}{\tau}\ \langle v-f , \zeta_n \rangle_2 + \frac{\lambda_n^2}{2\tau}\ \|\zeta_n\|_2^2\,. \label{c13} 
\end{align}
Observing that $(\zeta_n)_{n\ge 1}$ converges towards $\zeta_\infty:=\gamma_\eta(\zeta)\ \nu_\eta + \zeta$ in $H^2(\Omega)$ as $n\to \infty$, we find that
\begin{align*}
\lim_{n\to\infty} \big\langle \Delta\zeta_n , W'(v+\lambda_n \zeta_n) - \Delta v \big\rangle_2 & = \int_\Omega \mu \ \Delta\zeta_\infty\ dx\,,\\
\lim_{n\to\infty} \lambda_n\ \|\zeta_n\|_2^2 = \lim_{n\to\infty} \lambda_n\ \|\Delta\zeta_n\|_2^2 &= 0\,,
\end{align*}
while classical arguments and the embedding of $H^2(\Omega)$ in $L^\infty(\Omega)$ ensure that
\begin{align*}
& \lim_{n\to\infty} \frac{1}{\lambda_n}\ \big\langle \Delta v , W'(v+\lambda_n \zeta_n) - W'(v) \big\rangle_2 = \int_\Omega W''(v)\ \zeta_\infty\ \Delta v\ dx\,, \\
& \lim_{n\to\infty} \frac{1}{2\lambda_n}\ \big\langle W'(v+\lambda_n \zeta_n) - W'(v) , W'(v+\lambda_n \zeta_n) + W'(v) \big\rangle_2\\ 
& \qquad\qquad = \int_\Omega W'(v)\ W''(v)\ \zeta_\infty\ dx\,.
\end{align*}
Dividing \eqref{c13} by $\lambda_n$ and letting $n\to \infty$ give
\begin{align*}
0 \le & - \int_\Omega W''(v)\ \zeta_\infty\ \Delta v\ dx - \int_\Omega \mu \ \Delta\zeta_\infty\ dx + \int_\Omega W'(v)\ W''(v)\ \zeta_\infty\ dx + \frac{\langle v-f , \zeta_\infty \rangle_2}{\tau} \,,\\
= & \int_\Omega \mu \left[ - \Delta\left( \gamma_\eta(\zeta)\ \nu_\eta + \zeta \right) + W''(v)\left( \gamma_\eta(\zeta)\ \nu_\eta + \zeta \right) \right]\ dx + \int_\Omega \frac{v-f}{\tau} \left( \gamma_\eta(\zeta)\ \nu_\eta + \zeta \right)\ dx\,.
\end{align*}
Observing that $-\zeta$ also belongs to $H_N^2(\Omega)\cap\mathcal{C}^\infty(\overline{\Omega})$ and satisfies $\overline{(-\zeta)}=0$, the above inequality is also valid for $-\zeta$ and, since $\gamma_\eta(-\zeta)=-\gamma_\eta(\zeta)$, we end up with  
\begin{equation}
\int_\Omega \mu \left[ - \Delta\left( \gamma_\eta(\zeta)\ \nu_\eta + \zeta \right) + W''(v)\left( \gamma_\eta(\zeta)\ \nu_\eta + \zeta \right) \right]\ dx + \int_\Omega \frac{v-f}{\tau} \left( \gamma_\eta(\zeta)\ \nu_\eta + \zeta \right)\ dx=0 \label{c14}
\end{equation}
for all $\zeta\in H_N^2(\Omega)\cap\mathcal{C}^\infty(\overline{\Omega})$ satisfying $\overline{\zeta}=0$.

A first consequence of \eqref{c14} is that $\mu\in H^2_N (\Omega)$: indeed, we can also write \eqref{c14} as
\begin{align*}
\int_\Omega \mu\ \Delta\zeta\ dx = & -\gamma_\eta(\zeta)\ \langle \mu , \Delta\nu_\eta \rangle_2 + \gamma_\eta(\zeta)\ \left\langle W''(v)\mu + \frac{v-f}{\tau} , \nu_\eta \right\rangle_2 \\
& + \left\langle W''(v)\mu + \frac{v-f}{\tau} , \zeta \right\rangle_2\,.
\end{align*}
Since $v\in H^2(\Omega)$, $\mu\in L^2(\Omega)$, and
$$
|\gamma_\eta(\zeta)|\le \frac{\|\mu\|_2\ \|\zeta\|_2}{\langle \mu , \nu_\eta \rangle_2}\,,
$$
it follows from the continuous embedding of $H^2(\Omega)$ in $L^\infty(\Omega)$ that
\begin{eqnarray*}
\left| \int_\Omega \mu\ \Delta\zeta\ dx \right| & = & \frac{\|\mu\|_2\ \|\zeta\|_2}{\langle \mu , \nu_\eta \rangle_2} \left[ \|\mu\|_2\ \|\nu_\eta\|_{H^2} + \left( \|W''(v)\|_\infty\ \|\mu\|_2 + \frac{\|v-f\|_2}{\tau} \right)\ \|\nu_\eta\|_2  \right] \\
& & \qquad + \left( \| W''(v)\|_\infty\ \|\mu\|_2 + \frac{\|v-f\|_2}{\tau} \right) \ \|\zeta\|_2 \\
& \le & C(\eta,v,f,\tau)\ \|\zeta\|_2\,,
\end{eqnarray*}
and a two-step duality argument entails first that $\mu\in H^2(\Omega)$ and then that it satisfies the Neumann homogeneous boundary conditions, that is, 
\begin{equation}
\mu\in H_N^2(\Omega)\,. \label{c15}
\end{equation}
We can then integrate twice by parts the first term of the left-hand side of \eqref{c14} to obtain
\begin{equation*}
\int_\Omega \Xi\ \left( \gamma_\eta(\zeta)\ \nu_\eta + \zeta \right)\ dx =0 
\end{equation*}
for all $\zeta\in H_N^2(\Omega)\cap\mathcal{C}^\infty(\overline{\Omega})$ satisfying $\overline{\zeta}=0$, where
\begin{equation*}
\Xi := \frac{v-f}{\tau} - \Delta\mu + W''(v)\ \mu\,.
\end{equation*}
We may now let $\eta\to 0$ with the help of  \eqref{c12b}  and use a density argument to conclude that 
\begin{equation*}
\int_\Omega \Xi\ \left( \zeta - \frac{\langle \mu , \zeta \rangle_2}{\|\mu-\overline{\mu}\|_2^2}\ (\mu-\overline{\mu}) \right)\ dx =0 
\end{equation*}
for all $\zeta\in L^2(\Omega)$ satisfying $\overline{\zeta}=0$. Alternatively, 
\begin{equation*}
\int_\Omega \left( \Xi - \frac{\langle \Xi , \mu-\overline{\mu} \rangle_2}{\|\mu-\overline{\mu}\|_2^2}\ \mu \right)\ \zeta\ dx =0\,.
\end{equation*}
Now, take $\zeta\in L^2(\Omega)$. Applying the above identity to $\zeta-\overline{\zeta}$ gives
\begin{equation*}
\int_\Omega \left( \Xi - \overline{\Xi} - \frac{\langle \Xi , \mu-\overline{\mu} \rangle_2}{\|\mu-\overline{\mu}\|_2^2}\ (\mu-\overline{\mu}) \right)\ \zeta\ dx =0\,,
\end{equation*}
and, since this equality is valid for all test functions in $L^2(\Omega)$, we realize that 
\begin{equation*}
\frac{v-f}{\tau} - \Delta\mu + W''(v)\ \mu = \Xi = A + B\ \mu \;\;\mbox{ in }\;\; \Omega\,, 
\end{equation*}
with
\begin{equation}
A := \overline{\Xi} - \frac{\langle \Xi , \mu-\overline{\mu} \rangle_2}{\|\mu-\overline{\mu}\|_2^2}\ \overline{\mu} \;\;\mbox{ and }\;\; B:= \frac{\langle \Xi , \mu-\overline{\mu} \rangle_2}{\|\mu-\overline{\mu}\|_2^2}\,.\label{c16}
\end{equation}
The proof of Lemma~\ref{le:c2} is complete.
\end{proof}

Noting that the Lagrange multipliers $A$ and $B$ arising in Lemma~\ref{le:c2} are defined in a somewhat implicit way according to \eqref{c16}, the next step is to obtain some estimates on both of them. As we shall see now, this is quite easy for $A+B\overline{\mu}$ for which we have an explicit formula but turns out to be more complicated for $B$. 

\begin{lemma}\label{le:c3}
Assume that $v\in\mmab^2$ solves the minimization problem \eqref{c1} and is such that $\mu:= -\Delta v + W'(v)$ is not a constant. Assume further that $\overline{f}=\alpha$ and $\mzab=\emptyset$ and consider a positive real number $M\ge E(f)$. Then there is a positive real number $\kappa_M>0$ depending on $M$ such that
\begin{equation}
\left| A + B\ \overline{\mu} \right| \le \kappa_M \;\;\mbox{ and }\;\; |B| \le \kappa_M \left( 1 + \frac{\| v-f\|_2}{\tau} \right)\,. \label{c21}
\end{equation}
\end{lemma}
 
\begin{proof}
Integrating \eqref{c17} over $\Omega$ and using the properties $\mu\in H_N^2(\Omega)$ and $\overline{v}=\overline{f}=\alpha$, we obtain the identity
$$
A + B\ \overline{\mu} = \overline{W''(v)\ \mu}\,.
$$
It then follows from the H\"older inequality, the continuous embedding of $H^1(\Omega)$ in $L^4(\Omega)$, \eqref{b6}, $F(v)=\beta$, and \eqref{c1b} that 
\begin{align*}
\left| A + B\ \overline{\mu} \right| & \le C\ \|W''(v)\|_2\ \|\mu\|_2 \le C \left( 1 + \|v\|_4^2 \right)\ \sqrt{E[v]} \\
& \le C \left( 1 + \|v\|_{H^1}^2 \right)\ \sqrt{E[f]} \le C\ \sqrt{M}\,,
\end{align*}
which is the first inequality in \eqref{c21}. Next, since the average of $\mu-\overline{\mu}$ over $\Omega$ is zero, we can apply the operator $\mathcal{N}$ introduced in \eqref{b10} on $\mu-\overline{\mu}$ and the function $\mathcal{N}\left( \mu-\overline{\mu} \right)$ belongs to $H^2_N(\Omega)$. We then infer from \eqref{b10} and \eqref{c17} that
\begin{align}
B \left\|\nabla \mathcal{N}\left( \mu-\overline{\mu} \right) \right\|_2^2 & = - B\ \int_\Omega \mathcal{N}\left( \mu-\overline{\mu} \right)\ \Delta\mathcal{N}\left( \mu-\overline{\mu} \right)\ dx = B\ \int_\Omega \left( \mu-\overline{\mu} \right)\ \mathcal{N}\left( \mu-\overline{\mu} \right)\ dx \nonumber\\
& = \int_\Omega \left( \frac{v-f}{\tau} - \Delta\mu + W''(v)\ \mu - A \right)\ \mathcal{N}\left( \mu-\overline{\mu} \right)\ dx \nonumber\\
& = \int_\Omega \left( \frac{v-f}{\tau} + W''(v)\ \mu \right)\ \mathcal{N}\left( \mu-\overline{\mu} \right)\ dx + \left\| \mu-\overline{\mu} \right\|_2^2 \,. \label{c23}
\end{align}
On the one hand, since $\mzab=\emptyset$, $v\in\mmab^2$, and $E[v]\le E[f] \le M$ by \eqref{c1b}, we can apply Lemma~\ref{le:b3} which simply yields 
\begin{equation}
|B| \left\|\nabla \mathcal{N}\left( \mu-\overline{\mu} \right) \right\|_2^2 \ge m_M\ |B|\,. \label{c24}
\end{equation}
On the other hand, according to \eqref{c1b}, we have
\begin{equation}
\left\| \mu-\overline{\mu} \right\|_2 \le \left\| \mu\right\|_2 + \left\|\overline{\mu} \right\|_2 = 2\|\mu\|_2 = \sqrt{8 E[v]} \le C\ \sqrt{E[f]} \le C\ \sqrt{M} \,, \label{c25}
\end{equation}
which, together with Lemma~\ref{le:b2}, the Poincar\'e-Wirtinger inequality \eqref{b8}, the embedding of $H^2(\Omega)$ in $L^\infty(\Omega)$, \eqref{b10}, \eqref{c1b}, and \eqref{c25}, gives
\begin{align}
\left| \int_\Omega \left( \frac{v-f}{\tau} + W''(v)\ \mu \right)\ \mathcal{N}\left( \mu-\overline{\mu} \right)\ dx \right| & \le \left( \frac{\| v-f\|_2}{\tau} + \|W''(v)\|_\infty\ \|\mu\|_2 \right)\ \left\| \mathcal{N}\left( \mu-\overline{\mu} \right) \right\|_2 \nonumber\\
\le\ & C_2 \left( \frac{\| v-f\|_2}{\tau} + C\ (1 + \|v\|_\infty^2)\ \sqrt{2 E[v]} \right)\ \left\| \nabla \mathcal{N}\left( \mu-\overline{\mu} \right) \right\|_2 \nonumber \\
\le\ & C \left( \frac{\| v-f\|_2}{\tau} + (1 + \|v\|_{H^2}^2)\ \sqrt{E[f]} \right)\ \left\| \mu-\overline{\mu} \right\|_2 \nonumber\\
\le\ & C \left( \frac{\| v-f\|_2}{\tau} + (1 + E[v])\ \sqrt{M} \right)\ \sqrt{M} \nonumber\\
\le\ & C \left( \frac{\| v-f\|_2}{\tau} + (1 + E[f])\ \sqrt{M} \right)\ \sqrt{M} \nonumber\\
\le\ & C(M) \left( 1 + \frac{\| v-f\|_2}{\tau} \right)\,. \label{c26}
\end{align} 
We then infer from \eqref{c23}--\eqref{c26} that 
\begin{align*}
m_M\ |B| & \le |B| \left\|\nabla \mathcal{N}\left( \mu-\overline{\mu} \right) \right\|_2^2 \\
& \le \left| \int_\Omega \left( \frac{v-f}{\tau} + W''(v)\ \mu \right)\ \mathcal{N}\left( \mu-\overline{\mu} \right)\ dx \right| + \left\| \mu-\overline{\mu} \right\|_2^2 \\
& \le C(M) \left( 1 + \frac{\| v-f\|_2}{\tau} \right)\,,
\end{align*}
which readily gives the second inequality in \eqref{c21} since $m_M>0$ by Lemma~\ref{le:b3}.
\end{proof}

\section{Existence}\label{sec:exist}

In this section, we prove the existence part of Theorem~\ref{th:a1}. We fix $\alpha\in\RR$ and $\beta\in (\beta_\alpha,\infty)$ such that 
\begin{equation}
\mzab = \emptyset\,. \label{d0}
\end{equation}
Consider an initial condition $v_0\in \mmab^2$ and a time step $\tau\in (0,1)$. We define a sequence $(v_n^\tau)_{n\ge 1}$ inductively as follows:
\begin{align}
& v_0^\tau  = v_0\,,\label{d20} \\
& v_{n+1}^\tau  \mbox{ is a minimizer of }\; \mf_{\tau,v_n^\tau} \;\mbox{ in }\; \mmab^2\,, \qquad n\ge 0\,,\label{d21}
\end{align}
the functional $\mf_{\tau,v_n^\tau}$ being defined in \eqref{c0}. Owing to Lemma~\ref{le:c1}, this sequence is well defined. Setting 
\begin{equation}
\mu_n^\tau := - \Delta v_n^\tau + W'\left( v_n^\tau \right)\,, \qquad n\ge 0\,, \label{d22}
\end{equation}
we define two piecewise constant time-dependent functions $v^\tau$ and $\mu^\tau$ by 
\begin{equation}
\left( v^\tau(t) , \mu^\tau(t) \right) := \left( v_n^\tau , \mu_n^\tau \right)\,, \qquad t\in [n\tau, (n+1)\tau)\,, \qquad n\ge 0\,. \label{d23}
\end{equation}
From the minimizing property \eqref{d21} of $v_{n+1}^\tau$, $n\ge 0$, we deduce the following estimates on $\left( v^\tau , \mu^\tau \right)$.

\begin{lemma}\label{le:d1}
For $\tau\in (0,1)$, $t_1\ge 0$, and $t_2>t_1$, we have
\begin{align}
E\left[ v^\tau(t_2) \right] & \le E\left[ v^\tau(t_1) \right] \le E[v_0]\,, \label{d24} \\
\left\| v^\tau(t_1) \right\|_{H^2} + \left\| \mu^\tau(t_1) \right\|_2 & \le C_1\ \left( 1+\sqrt{E[v_0]} \right)\,, \label{d24b} \\
\sum_{n=0}^\infty \left\| v_{n+1}^\tau - v_n^\tau \right\|_2^2 & \le 2\tau\ E[v_0]\,, \label{d27} \\
\left\| v^\tau(t_2) - v^\tau(t_1) \right\|_2^2 & \le 2E[v_0]\ (\tau + t_2-t_1)\,. \label{d25}
\end{align}
\end{lemma}
 
\begin{proof}
Consider $n\ge 0$. According to the definition \eqref{d21} of $v_{n+1}^\tau$, we have $\mf_{\tau,v_n^\tau}\left[ v_{n+1}^\tau \right] \le \mf_{\tau,v_n^\tau}\left[ v_n^\tau \right]$, that is,
\begin{equation}
\frac{\left\| v_{n+1}^\tau - v_n^\tau \right\|_2^2}{2\tau} + E\left[ v_{n+1}^\tau \right] \le E\left[ v_n^\tau \right]\,. \label{d26}
\end{equation}
On the one hand, the time monotonicity \eqref{d24} of $t\mapsto E\left[ v^\tau(t) \right]$ readily follows from \eqref{d26} by induction while the $H^2$-estimate on $v^\tau$ and the $L^2$-estimate on $\mu^\tau$ in \eqref{d24b} are straightforward consequences of Lemma~\ref{le:b2}, the definition of $E$ and $\mu^\tau$, and \eqref{d24}. On the other hand, summing \eqref{d26} over $n\ge 0$ gives \eqref{d27}. 

Finally, let $t_2>t_1\ge 0$ and denote the largest integer smaller than $t_i/\tau$ by $n_i$, $i=1,2$. We then infer from \eqref{d27} that 
\begin{align*}
\left\| v^\tau(t_2) - v^\tau(t_1) \right\|_2 & = \left\| v_{n_2}^\tau - v_{n_1}^\tau \right\|_2 \le \sum_{n=n_1}^{n_2-1} \left\| v_{n+1}^\tau - v_n^\tau \right\|_2 \le \sqrt{n_2-n_1} \left( \sum_{n=n_1}^{n_2-1} \left\| v_{n+1}^\tau - v_n^\tau \right\|_2^2 \right)^{1/2} \\
& \le \left( \frac{t_2+\tau - t_1}{\tau} \right)^{1/2}\ \sqrt{2\tau E[v_0]}\,,
\end{align*}
whence \eqref{d25}.
\end{proof}

Consider next $n\ge 0$. We observe that, since $v_{n+1}^\tau\in H^2_N(\Omega)$, the function $\mu_{n+1}^\tau$ defined in \eqref{d22} is constant if and only if $-\Delta v_{n+1}^\tau + W'\left( v_{n+1}^\tau \right) - \overline{W'\left( v_{n+1}^\tau \right)} = 0$ in $\Omega$, that is, $v_{n+1}^\tau\in\mzab$. Since $\mzab$ is assumed to be empty, this situation cannot occur and we have thus established that $\mu_{n+1}^\tau$ is not a constant for all $n\ge 0$. We are then in a position to apply Lemma~\ref{le:c2} for each $n\ge 0$ and deduce that $\mu_{n+1}^\tau\in H_N^2(\Omega)$ and there are real numbers $A_{n+1}^\tau$ and $B_{n+1}^\tau$ such that $\mu_{n+1}^\tau$ solves
\begin{equation}
\frac{v_{n+1}^\tau - v_n^\tau}{\tau} - \Delta\mu_{n+1}^\tau + W''\left( v_{n+1}^\tau \right)\ \mu_{n+1}^\tau = A_{n+1}^\tau + B_{n+1}^\tau\ \mu_{n+1}^\tau \;\;\mbox{ in }\;\; \Omega\,. \label{d33}
\end{equation} 
We then define two piecewise constant time-dependent functions $A^\tau$ and $B^\tau$ by
\begin{equation}
\left( A^\tau(t) , B^\tau(t) \right) := \left( A_n^\tau , B_n^\tau \right)\,, \qquad t\in [n\tau, (n+1)\tau)\,, \qquad n\ge 1\,, \label{d34}
\end{equation}
and collect bounds for these functions in the next lemma.

\begin{lemma}\label{le:d2}
For $\tau\in (0,1)$ and $T>\tau$, we have
\begin{equation}
\int_\tau^T \left( \left| A^\tau(t) \right|^2 + \left| B^\tau(t) \right|^2 +\left\| \mu^\tau(t) \right\|_{H^2}^2\right)\ dt \le C_3(T)\,. \label{d35}
\end{equation}
\end{lemma}

\begin{proof}
Owing to \eqref{d0} and the time monotonicity \eqref{d24} of $E\left[ v^\tau \right]$, the assumptions of Lemma~\ref{le:c3} are satisfied with $M=E[v_0]$ for all $n\ge 0$ and we obtain the estimates:
\begin{equation}
\left| A_{n+1}^\tau + B_{n+1}^\tau\ \overline{\mu_{n+1}^\tau} \right| \le C \;\;\mbox{ and }\;\; \left| B_{n+1}^\tau \right| \le C \left( 1 + \frac{\left\| v_{n+1}^\tau - v_n^\tau \right\|_2}{\tau} \right)\,, \qquad n\ge 0\,. \label{d36}
\end{equation}
Let $T>\tau$ and $m\ge 1$ be the largest integer smaller than $T/\tau$. On the one hand, we infer from \eqref{d27} and \eqref{d36} that 
\begin{align}
\int_\tau^T \left| B^\tau(t) \right|^2\ dt & \le \tau\ \sum_{n=0}^{m-1} \left| B_{n+1}^\tau \right|^2 \le C \tau\ \sum_{n=0}^{m-1} \left( 1 + \frac{\left\| v_{n+1}^\tau - v_n^\tau \right\|_2^2}{\tau^2} \right) \nonumber \\
& \le C \left( m \tau + 2 E[v_0] \right) \le C\ (1+T)\,. \label{d37}
\end{align}
On the other hand, since $v_{n+1}^\tau \in H_N^2(\Omega)$, we have $\overline{\mu_{n+1}^\tau} = \overline{W'\left( v_{n+1}^\tau \right)}$ and, since $v_{n+1}^\tau \in\mmab^2$,  it follows from the continuous embedding of $H^1(\Omega)$ in $L^3(\Omega)$ and Lemma~\ref{le:b2} that 
$$
\left| \overline{\mu_{n+1}^\tau} \right| = \left| \overline{W'\left( v_{n+1}^\tau \right)} \right| \le C \left( 1 + \left\| v_{n+1}^\tau \right\|_3^3 \right) \le C \left( 1 + \left\| v_{n+1}^\tau \right\|_{H^1}^3 \right) \le C\,.
$$
Consequently, thanks to \eqref{d36} and \eqref{d37}, we have that
\begin{align}
\int_\tau^T \left| A^\tau(t) \right|^2\ dt & \le \tau\ \sum_{n=0}^{m-1} \left| A_{n+1}^\tau \right|^2 \le 2 \tau\ \sum_{n=0}^{m-1} \left( \left| A_{n+1}^\tau + B_{n+1}^\tau\ \overline{\mu_{n+1}^\tau} \right|^2 + \left| B_{n+1}^\tau\ \overline{\mu_{n+1}^\tau} \right|^2 \right) \nonumber \\
& \le C\tau \left( m + \sum_{n=0}^{m-1} \left| B_{n+1}^\tau \right|^2 \right) \le C\ (1+T)\, . \label{d37b}
\end{align}
Finally, we observe that
\begin{equation}
\int_\tau^T \left\| \Delta\mu^\tau(t) \right\|_2^2\ dt \le \tau\ \sum_{n=0}^{m-1} \left\| \Delta\mu_{n+1}^\tau \right\|_2^2\,, \label{d38}
\end{equation}
an , using \eqref{d33} and the continuous embedding of $H^2(\Omega)$ in $L^\infty(\Omega)$, that
\begin{align*}
\left\| \Delta\mu_{n+1}^\tau \right\|_2^2 & \le C \left( \frac{\left\| v_{n+1}^\tau - v_n^\tau \right\|_2^2}{\tau^2} + \left\| W''\left( v_{n+1}^\tau \right)\ \mu_{n+1}^\tau  \right\|_2^2 + \left| A_{n+1}^\tau  \right|^2 + \left\| B_{n+1}^\tau\ \mu_{n+1}^\tau \right\|_2^2 \right) \\
& \le C \left( \frac{\left\| v_{n+1}^\tau - v_n^\tau \right\|_2^2}{\tau^2} + \left\| W''\left( v_{n+1}^\tau \right)\right\|_\infty^2  \left\| \mu_{n+1}^\tau  \right\|_2^2 + \left| A_{n+1}^\tau  \right|^2 + \left| B_{n+1}^\tau \right|^2 \left\| \mu_{n+1}^\tau \right\|_2^2 \right) \\
& \le C \left( \frac{\left\| v_{n+1}^\tau - v_n^\tau \right\|_2^2}{\tau^2} + \left( 1 + \left\| v_{n+1}^\tau \right\|_{H^2}^4 \right)  \left\| \mu_{n+1}^\tau  \right\|_2^2 + \left| A_{n+1}^\tau  \right|^2 + \left| B_{n+1}^\tau \right|^2 \left\| \mu_{n+1}^\tau \right\|_2^2 \right)\,;
\end{align*}
hence, thanks to \eqref{d24b}, we have 
\begin{equation*}
\left\| \Delta\mu_{n+1}^\tau \right\|_2^2 \le C \left( 1 + \frac{\left\| v_{n+1}^\tau - v_n^\tau \right\|_2^2}{\tau^2}  + \left| A_{n+1}^\tau  \right|^2 + \left| B_{n+1}^\tau \right|^2 \right)
\end{equation*}
for $n\in\{0, \ldots, m-1\}$. We then infer from \eqref{d27}, \eqref{d37}, \eqref{d37b}, \eqref{d38}, and the above estimate that
$$
\int_\tau^T \left\| \Delta\mu^\tau(t) \right\|_2^2\ dt \le C\tau\, \left( m + \frac{2\tau E[v_0]}{\tau^2} \right) + \int_\tau^{T+\tau}  \left( \left| A^\tau(t) \right|^2 + \left| B^\tau(t) \right|^2 \right)\ dt  \le C(T)\,.
$$
Combining this estimate with \eqref{d37}, \eqref{d37b}, and the $L^2$-bound \eqref{d24b} on $\mu^\tau$ gives \eqref{d35}.
\end{proof}

Thanks to the above analysis, all the tools required to perform the limit as $\tau\to 0$ are now available and we may thus proceed to identify the behaviour of $(v^\tau)$ as $\tau\to 0$. We begin with a consequence of \eqref{d24b}  which guarantees compactness with respect to the space variable and \eqref{d25} which gives the time equicontinuity: owing to \eqref{d24b}, \eqref{d25}, and the compactness of the embedding of $H^2(\Omega)$ in $H^1(\Omega)$ and $\mathcal{C}\left(\overline{\Omega}\right)$, a refined version of the Ascoli-Arzel\`a theorem \cite[Proposition~3.3.1]{AGS08} ensures that there are a subsequence $\left( v^{\tau_k} \right)_{k\ge 1}$ of $(v^\tau)$ and a function
\begin{equation}
v\in \mathcal{C}\left( [0,\infty)\times\overline{\Omega} \right)\cap \mathcal{C}([0,\infty);H^1(\Omega))\cap L^\infty(0,\infty;H^2(\Omega)) \label{d28b}
\end{equation}
such that 
\begin{equation}
v^{\tau_k}(t) \longrightarrow v(t) \;\;\mbox{ in }\;\; \mathcal{C}\left(\overline{\Omega}\right)\cap H^1(\Omega) \;\;\mbox{ for all }\;\; t\ge 0\,. \label{d28}
\end{equation}
A straightforward consequence of \eqref{d24b}, \eqref{d28}, the continuity of the embedding of $H^2(\Omega)$ in $H^1(\Omega)$ and $\mathcal{C}\left(\overline{\Omega}\right)$, and Lebesgue's dominated convergence theorem is that
\begin{equation}
v^{\tau_k} \longrightarrow v \;\;\mbox{ in }\;\; L^p(0,T;\mathcal{C}\left( \overline{\Omega}\right)\cap H^1(\Omega)) \;\;\mbox{ for all }\;\; p\in [1,\infty) \;\;\mbox{ and }\;\; T>0\,. \label{d28a}
\end{equation}
In addition, it follows from \eqref{d24b} and \eqref{d35} that we may assume that there are functions
$$
\mu\in L^\infty(0,\infty;L^2(\Omega))\cap L_{\text{loc}}^2(0,\infty;H^2(\Omega))\,, \qquad A \in L_{\text{loc}}^2(0,\infty)\,, \qquad B \in L_{\text{loc}}^2(0,\infty)\,,
$$
such that, for all $T>\delta>0$,
\begin{align}
v^{\tau_k} \stackrel{*}{\rightharpoonup} v & \;\;\mbox{ in }\;\; L^\infty(0,T;H^2(\Omega)) \,, \label{d29} \\
\mu^{\tau_k} \stackrel{*}{\rightharpoonup} \mu & \;\;\mbox{ in }\;\; L^\infty(0,T;L^2(\Omega)) \;\;\mbox{ and }\;\;
\mu^{\tau_k} \rightharpoonup \mu  \;\;\mbox{ in }\;\; L^2(\delta,T;H^2(\Omega)) \,, \label{d30} \\
A^{\tau_k} \rightharpoonup A & \;\;\mbox{ in }\;\; L^2(\delta,T) \;\;\mbox{ and }\;\; B^{\tau_k} \rightharpoonup B \;\;\mbox{ in }\;\; L^2(\delta,T) \,. \label{d370} 
\end{align}
Now, since $v^{\tau_k}(t)\in\mmab^2$ and $\mu^{\tau_k}(t) = -\Delta v^{\tau_k}(t) + W'\left( v^{\tau_k}(t) \right)$ for all $t\ge 0$, it readily follows from the convergences \eqref{d28}, \eqref{d29}, and \eqref{d30} that 
\begin{equation}
v(t)\in\mmab^2\;\;\mbox{ and }\;\; \mu(t) = -\Delta v(t) + W'(v(t))\;\;\mbox{ for all }\;\; t\ge 0 \,. \label{d31}
\end{equation}
It remains to derive the equation solved by $v$. To this end, we have to pass to the limit in \eqref{d33} and in particular to identify the limits of the nonlinear terms $\left( W''\left( v^{\tau_k} \right) \mu^{\tau_k} \right)_{k\ge 1}$ and $\left( B^{\tau_k} \mu^{\tau_k} \right)_{k\ge 1}$. Concerning the former, we combine the strong convergence \eqref{d28a} of $\left( v^{\tau_k} \right)_{k\ge 1}$ with the weak convergence \eqref{d30} of $\left( \mu^{\tau_k} \right)_{k\ge 1}$ to obtain that
\begin{equation}
W''\left( v^{\tau_k} \right) \mu^{\tau_k} \rightharpoonup W''(v)\ \mu \;\;\mbox{ in }\;\; L^2((0,T)\times \Omega) \;\;\mbox{ for all }\;\; T>0\,. \label{d400}
\end{equation}
As for the latter, the situation is less clear as the convergences \eqref{d30} and \eqref{d370} are both weak convergences. However, we take advantage at this point of the fact that $B^{\tau_k}$ depends only on time. Indeed, on the one hand, we notice that the strong convergence \eqref{d28a} of $\left( v^{\tau_k} \right)_{k\ge 1}$ implies that $\left( \mu^{\tau_k} \right)_{k\ge 1}$ converges strongly towards $\mu$ in $L^p(0,T;H^1(\Omega)')$ for all $p\in [1,\infty)$ and $T>0$. We combine this convergence with \eqref{d370} to obtain that $\left( B^{\tau_k} \mu^{\tau_k} \right)_{k\ge 1}$ converges towards $B\mu$ in the sense of distributions. On the other hand, the sequence $\left( B^{\tau_k} \mu^{\tau_k} \right)_{k\ge 1}$ is bounded in $L^2((\delta,T)\times \Omega)$ by \eqref{d24b} and \eqref{d35} for all $T>\delta>0$ and is thus weakly compact in that space. Therefore, we have shown that, after possibly extracting a further subsequence, 
\begin{equation}
B^{\tau_k}\ \mu^{\tau_k} \rightharpoonup B\ \mu \;\;\mbox{ in }\;\; L^2((\delta,T)\times \Omega) \;\;\mbox{ for all }\;\; T>\delta>0\,. \label{d401}
\end{equation}
Now, for $t_2>t_1>0$ and $\varphi\in L^2(\Omega)$, we denote the largest integer smaller than $t_i/\tau_k$ by $n_{i,k}$, $i = 1, 2$, and infer from \eqref{d33} that 
\begin{align*}
\int_\Omega \left( v^{\tau_k}(t_2) \right. & - \left. v^{\tau_k}(t_1) \right)\ \varphi\ dx = \int_\Omega \left( v_{n_{2,k}}^{\tau_k} - v_{n_{1,k}}^{\tau_k} \right)\ \varphi\ dx = \sum_{n=n_{1,k}+1}^{n_{2,k}} \int_\Omega \left( v_{n}^{\tau_k} - v_{n-1}^{\tau_k} \right)\ \varphi\ dx \\ 
& = \tau_k\ \sum_{n=n_{1,k}+1}^{n_{2,k}} \int_\Omega \left( \Delta\mu_{n}^{\tau_k} - W''\left( v_{n}^{\tau_k} \right)\ \mu_{n}^{\tau_k} + A_{n}^{\tau_k} + B_{n}^{\tau_k}\ \mu_{n}^{\tau_k} \right)\ \varphi\ dx \\ 
& = \sum_{n=n_{1,k}+1}^{n_{2,k}} \int_{n\tau_k}^{(n+1)\tau_k} \int_\Omega \left( \Delta\mu^{\tau_k} - W''\left( v^{\tau_k} \right)\ \mu^{\tau_k} + A^{\tau_k} + B^{\tau_k}\ \mu^{\tau_k} \right)(t)\ \varphi\ dxdt \\ 
& = \int_{(n_{1,k}+1)\tau_k}^{(n_{2,k}+1)\tau_k} \int_\Omega \left( \Delta\mu^{\tau_k} - W''\left( v^{\tau_k} \right)\ \mu^{\tau_k} + A^{\tau_k} + B^{\tau_k}\ \mu^{\tau_k} \right)(t)\ \varphi\ dxdt
\end{align*}
Clearly, $(n_{i,k}+1)\tau_k\to t_i$ as $k\to\infty$, $i=1,2$. Letting $k\to\infty$ in the above identity gives, thanks to \eqref{d28}, \eqref{d30}, \eqref{d400}, \eqref{d370}, and \eqref{d401}, 
\begin{equation}
\int_\Omega \left( v(t_2) - v(t_1) \right)\ \varphi\ dx = \int_{t_1}^{t_2} \int_\Omega \left( \Delta\mu - W''(v)\ \mu + A + B\ \mu \right)(t)\ \varphi\ dxdt\,. \label{d39}
\end{equation}
It is now straightforward to check that Lemma~\ref{le:d2} and the convergences \eqref{d30} and \eqref{d370} imply that $\mu\in L^2(0,T;H^2(\Omega))$, $A\in L^2(0,T)$, and $B\in L^2(0,T)$ for all $T>0$. Combining these integrability properties with \eqref{d28} ensures that \eqref{d39} is also valid for $t_1=0$. The regularity of $v$ and $\mu$ then allows us to deduce \eqref{a4} from \eqref{d39}.

It remains to check that $A$ and $B$ are given by \eqref{a1} and \eqref{a2}, respectively: first, \eqref{a1} readily follows by integrating \eqref{a4} and using that $\overline{v(t)}=\alpha$ for all $t\ge 0$ and the homogeneous Neumann boundary conditions satisfied by $\mu$. Next, since $F[v(t)]=\beta$ for all $t>0$, we differentiate this identity with respect to time and, using once more the homogeneous Neumann boundary conditions for $\mu$ and the fact that 
$$
 \int_\Omega  \overline{\mu}  \left( \mu - \overline{\mu}\right) \\ dx =0 \, ,
$$
we obtain
\begin{equation*}
0 = \int_\Omega \mu\ \partial_t v \ dx = - \|\nabla\mu\|_2^2 - \int_\Omega W''(v)\ \mu^2\ dx + \left( A + B \overline{\mu} \right) \int_\Omega \mu\ dx + B\ \|\mu - \overline{\mu}\|_2^2\,.
\end{equation*}
The identity \eqref{a2} now follows from the above identity with the help of \eqref{a1}.

Finally, fix $t>0$ and assume for contradiction that there is a sequence $(s_n)_{n\ge 1}$ in $[0,t]$ such that $\|(\mu - \overline{\mu})(s_n)\|_2\to 0$ as $n\to \infty$. Since $[0,t]$ is compact, we may assume that $s_n\to s_\infty$ as $n\to \infty$ for some $s_\infty\in [0,t]$. Thanks to the regularity of $v$, we actually have $v\in \mathcal{C}([0,t];H^1(\Omega)\cap \mathcal{C}(\overline{\Omega}))$ so that $((\mu - \overline{\mu})(s_n))_{n\ge 1}$ converges towards $(\mu - \overline{\mu})(s_\infty)$ in $H^1(\Omega)'$. Since it also converges to zero in $L^2(\Omega)$, we have shown that $(\mu - \overline{\mu})(s_\infty)=0$ which implies that $v(s_\infty)\in\mzab$ and contradicts \eqref{a0}. Therefore, $\|\mu - \overline{\mu} \|_2$ is bounded from below by a positive constant in $[0,t]$ and the proof of the existence part of Theorem~\ref{th:a1} is complete. 

\section{Uniqueness}\label{sec:uniq}

Consider $\alpha\in\RR$ and $\beta\in (\beta_\alpha,\infty)$ satisfying \eqref{a0}. Let $v_i$, $i=1,2$, be two solutions to \eqref{pf1}-\eqref{pf5} with $\mu_i := - \Delta v_i + W'(v_i)$ and associated Lagrange multipliers $(A_i,B_i)$, $i=1,2$. Owing to the regularity of $v_i$ and $\mu_i$, $i=1,2$, stated in Theorem~\ref{th:a1}, and the embedding of $H^2(\Omega)$ in $L^\infty(\Omega)$, the function $\phi$ defined by 
$$
\phi(t) := |B_1(t)| + |B_2(t)| + \|\mu_1(t)\|_{H^2} + \|\mu_2(t)\|_{H^2}\,, \qquad t>0\,,
$$
satisfies
\begin{equation}
\phi\in L^2(0,t) \;\;\mbox{ and }\;\; \|\mu_1(t)\|_\infty + \|\mu_2(t)\|_\infty \leq C \phi(t) \;\;\mbox{ for all }\;\; t>0\,. \label{u10} 
\end{equation}
Also, given $T>0$, Theorem~\ref{th:a1} (in particular \eqref{a3}) and the embedding of $H^2(\Omega)$ in $L^\infty(\Omega)$ ensure that there is $K_T>1$ such that
\begin{equation}
\|v_1(t)\|_\infty + \|v_2(t)\|_\infty + \|\mu_1(t)\|_2 + \|\mu_2(t)\|_2 \le K_T \;\;\mbox{ for all }\;\; t\in (0,T)\,, \label{u2}
\end{equation}
and 
\begin{equation}
\min{\left\{ \left\| \mu_1(t) - \overline{\mu_1(t)} \right\|_2 , \left\| \mu_2(t) - \overline{\mu_2(t)} \right\|_2 \right\}} \ge \frac{1}{K_T}> 0  \;\;\mbox{ for all }\;\; t\in (0,T)\,. \label{u3}
\end{equation}

Now, $v_1-v_2$ solves
\begin{equation}
\partial_t (v_1-v_2) - \Delta (\mu_1-\mu_2) = - W''(v_1)\ \mu_1 + W''(v_2)\ \mu_2 + A_1 - A_2 + B_1\ \mu_1 - B_2\ \mu_2 \label{u1}
\end{equation}
in $(0,\infty)\times\Omega$ with homogeneous Neumann boundary conditions for $v_1-v_2$ and $\mu_1-\mu_2$. As a first step of the uniqueness proof, we estimate some terms in the right-hand side of \eqref{u1}. A first consequence of \eqref{u2} is that
\begin{align}
\left\| W''(v_1)\ \mu_1 - W''(v_2)\ \mu_2 \right\|_2 \le & \left\| W''(v_1) \right\|_\infty \left\| \mu_1 - \mu_2 \right\|_2 + \left\| \mu_2 \right\|_\infty \left\| W''(v_1) - W''(v_2) \right\|_2 \nonumber\\
\le & \left\| W'' \right\|_{L^\infty(-K_T,K_T)} \left\| \mu_1 - \mu_2 \right\|_2 + C \phi \left\| W''' \right\|_{L^\infty(-K_T,K_T)} \left\| v_1-v_2 \right\|_2 \nonumber \\
\le & C(T) \left( \left\| \mu_1 - \mu_2 \right\|_2 + \phi \left\| v_1 - v_2 \right\|_2 \right)\,. \label{u4}
\end{align}
It then readily follows from \eqref{u4} that
\begin{equation}
\left| \overline{W''(v_1)\ \mu_1} - \overline{W''(v_2)\ \mu_2} \right| \le C(T) \left( \left\| \mu_1 - \mu_2 \right\|_2 + \phi \left\| v_1 - v_2 \right\|_2 \right)\,. \label{u5}
\end{equation}
We next estimate $B_1-B_2$. To this end, we first use \eqref{a2} to compute
\begin{equation}
\left| \left\| \mu_1 - \overline{\mu_1} \right\|_2^2\ B_1 - \left\| \mu_2 - \overline{\mu_2} \right\|_2^2\ B_2 \right| \le \sum_{j=1}^4 I_j\,, \label{u6a}
\end{equation}
where
\begin{align*}
I_1 := \left| \left\| \nabla\mu_1 \right\|_2^2 - \left\| \nabla\mu_2 \right\|_2^2 \right|\,, \quad & I_2 := \int_\Omega \left| W''(v_1)\ \mu_1^2 - W''(v_2)\ \mu_2^2 \right|\ dx\,, \\
I_3 := \|\mu_1\|_1  \left| \overline{W''(v_1)\ \mu_1} - \overline{W''(v_2)\ \mu_2} \right|\,, \quad & I_4 := \left| \overline{W''(v_2)\ \mu_2} \right|\ \|\mu_1-\mu_2\|_1
\end{align*}
Integrating by parts we find
\begin{align}
I_1 \le & \left| \int_\Omega \nabla(\mu_1+\mu_2) \cdot \nabla (\mu_1-\mu_2)\ dx \right| = \left| \int_\Omega (\mu_1-\mu_2)\ \Delta(\mu_1+\mu_2)\ dx \right| \nonumber \\
\le & \|\mu_1+\mu_2\|_{H^2}\ \|\mu_1-\mu_2\|_2 \le \phi\ \|\mu_1-\mu_2\|_2\,. \label{u6b}
\end{align}
We next deduce from \eqref{u2} and \eqref{u4} that 
\begin{align}
I_2 \le & \|\mu_1\|_2 \left\| W''(v_1)\ \mu_1 - W''(v_2)\ \mu_2 \right\|_2 + \left\| W''(v_2)\ \mu_2 \right\|_2 \|\mu_1-\mu_2\|_2 \nonumber \\
\le & C(T) \left( \left\| \mu_1 - \mu_2 \right\|_2 + \phi \left\| v_1 - v_2 \right\|_2 \right) + \left\| W''(v_2) \right\|_\infty \left\| \mu_2 \right\|_2 \|\mu_1-\mu_2\|_2 \nonumber \\
\le & C(T) \left( \left\| \mu_1 - \mu_2 \right\|_2 + \phi \left\| v_1 - v_2 \right\|_2 \right)\,. \label{u6c}
\end{align}
The last two terms are easier to estimate and we use \eqref{u2} and \eqref{u5} to obtain that 
\begin{align}
I_3 + I_4 \le & |\Omega|^{1/2}\ \|\mu_1\|_2 \left| \overline{W''(v_1)\ \mu_1} - \overline{W''(v_2)\ \mu_2} \right| + \left\| W''(v_2) \right\|_\infty \left\| \mu_2 \right\|_2 \|\mu_1-\mu_2\|_2 \nonumber \\
\le & C(T) \left( \left\| \mu_1 - \mu_2 \right\|_2 + \phi \left\| v_1 - v_2 \right\|_2 \right)\,. \label{u6d}
\end{align} 
Using \eqref{u2} and \eqref{u3}, we deduce from \eqref{u6a}--\eqref{u6d} that
\begin{align}
|B_1-B_2| \le & \frac{\left| \left\| \mu_1 - \overline{\mu_1} \right\|_2^2\ B_1 - \left\| \mu_2 - \overline{\mu_2} \right\|_2^2\ B_2 \right|}{\left\| \mu_1 - \overline{\mu_1} \right\|_2^2} + \frac{|B_2|}{\left\| \mu_1 - \overline{\mu_1} \right\|_2^2} \left| \left\| \mu_1 - \overline{\mu_1} \right\|_2^2 - \left\| \mu_2 - \overline{\mu_2} \right\|_2^2 \right| \nonumber \\
\le & K_T^2\ \sum_{j=1}^4 I_j + K_T^2\ \phi\ \left\| \mu_1 - \overline{\mu_1} + \mu_2 - \overline{\mu_2} \right\|_2\ \left\| \mu_1 - \overline{\mu_1} - \mu_2 + \overline{\mu_2} \right\|_2 \nonumber \\
\le & C(T)\ (1+\phi)\  \left( \left\| \mu_1 - \mu_2 \right\|_2 + \left\| v_1 - v_2 \right\|_2 \right)\,. \label{u7}
\end{align}

After this preparation, we multiply \eqref{u1} by $v_1-v_2$ and integrate over $\Omega$ to obtain
\begin{align*}
\frac{1}{2}\ \frac{d}{dt} \|v_1-v_2\|_2^2 - & \int_\Omega (\mu_1-\mu_2)\ \Delta (v_1-v_2)\ dx \\
= & - \int_\Omega (v_1-v_2)\ \left( W''(v_1)\ \mu_1 - W''(v_2)\ \mu_2 \right)\ dx \\
& + \left( A_1 - A_2 \right)\ \int_\Omega (v_1-v_2)\ dx \\
& + (B_1-B_2)\ \int_\Omega \mu_1\ (v_1-v_2)\ dx + B_2\ \int_\Omega (\mu_1-\mu_2)\ (v_1-v_2)\ dx\,.
\end{align*}
Since
\begin{align*}
- \int_\Omega (\mu_1-\mu_2)\ \Delta (v_1-v_2)\ dx = & \int_\Omega (\mu_1-\mu_2) \left( \mu_1 - W'(v_1) -\mu_2 + W'(v_2) \right)\ dx \\
\ge & \|\mu_1-\mu_2\|_2^2 - \|\mu_1-\mu_2\|_2\ \left\| W'(v_1)-W'(v_2) \right\|_2 \\
\ge & \frac{1}{2}\ \|\mu_1-\mu_2\|_2^2 - \frac{1}{2}\ \left\| W'' \right\|_{L^\infty(-K_T,K_T)}^2\ \|v_1-v_2\|_2^2 \\
\ge & \frac{1}{2}\ \|\mu_1-\mu_2\|_2^2 - C(T)\ \|v_1-v_2\|_2^2
\end{align*}
by \eqref{u2} and Young's inequality, we infer from  $\overline{v_1}=\overline{v_2}=\alpha$ that
\begin{align*}
\frac{1}{2}\ \frac{d}{dt} \|v_1-v_2\|_2^2 + & \frac{1}{2}\ \|\mu_1-\mu_2\|_2^2 - C(T)\ \|v_1-v_2\|_2^2 \\
& \le \| v_1-v_2\|_2\ \left\| W''(v_1)\ \mu_1 - W''(v_2)\ \mu_2 \right\|_2 + |B_1-B_2|\ \|\mu_1\|_2\ \|v_1-v_2\|_2 \\
& + |B_2|\ \|\mu_1-\mu_2\|_2\ \|v_1-v_2\|_2 \,.
\end{align*}
Using Young's inequality and \eqref{u2}, \eqref{u4}, and \eqref{u7}, we deduce 
\begin{align*}
\frac{d}{dt} \|v_1-v_2\|_2^2 & + \|\mu_1-\mu_2\|_2^2 \\
\le & C(T)\ \|v_1-v_2\|_2^2 + C(T)\ (1+\phi)\ \left( \left\| \mu_1 - \mu_2 \right\|_2 + \left\| v_1 - v_2 \right\|_2 \right)\ \|v_1-v_2\|_2 \\
& + \phi\ \|\mu_1-\mu_2\|_2\ \|v_1-v_2\|_2 \\
\le & \frac{1}{2}\ \|\mu_1-\mu_2\|_2^2 + C(T)\ \left( 1 + \phi^2 \right)\ \left\| v_1 - v_2 \right\|_2^2\,.
\end{align*}
Consequently, we obtain
$$
\frac{d}{dt} \left( \|(v_1-v_2)(t) \|_2^2 + \frac{1}{2}\ \int_0^t \|(\mu_1-\mu_2)(s) \|_2^2 \ ds \right) \le C(T)\ \left( 1 + \phi^2 (t) \right)\ \left\| (v_1 - v_2 \right)(t) \|_2^2 
$$
and, since $(1+ \phi^2)\in L^1(0,T)$ by \eqref{u10}, the Gronwall lemma gives
\begin{equation} \label{pier1}
\|(v_1-v_2)(t)\|_2^2 + \frac{1}{2}\ \int_0^t \|(\mu_1-\mu_2)(s) \|_2^2 \ ds \le C(T)\ \|(v_1-v_2)(0)\|_2^2  \quad \text{ for all }\;\; t\in [0,T]\,. 
\end{equation}
The uniqueness statement of Theorem~\ref{th:a1} then follows. 
\begin{remark}
By \eqref{pier1} we have actually established a property of continuous dependence
of the solutions to \eqref{pf1}-\eqref{pf5} on the initial data.
\end{remark}

\section*{Acknowledgements}

We warmly thank Xavier Cabr\'e and Giuseppe Savar\'e for illuminating discussions. The authors gratefully acknowledge financial support and kind hospitality of the Institut de Math\'ematiques de Toulouse, Universit\'e Paul Sabatier, and the IMATI of CNR in Pavia.

%
\bibliographystyle{abbrv}
\bibliography{PCPhL2}
%

\end{document}